\documentclass[12pt, reqno]{amsart}
\usepackage{amsmath, amsthm, amscd, amsfonts, amssymb, graphicx, color}
\usepackage[bookmarksnumbered, colorlinks, plainpages]{hyperref}
\usepackage{color}
\usepackage{mathptmx}	
\newtheorem{theorem}{Theorem}[section]

\newtheorem{proposition}{Proposition}[section]
\newtheorem{corollary}{Corollary}[section]
\theoremstyle{definition}
\newtheorem{definition}{Definition}[section]
\newtheorem{example}{Example}[section]

\theoremstyle{remark}
\newtheorem{remark}{Remark}[section]
\numberwithin{equation}{section}

\begin{document}

\title[\bf On $(n,m)$-$A$-normal and  $(n,m)$-$A$-quasinormal Semi-
Hilbert space operators
]{\bf On $(n,m)$-$A$-normal and  $(n,m)$-$A$-quasinormal  Semi-
Hilbert space operators}

\author[Samir Al Mohammady; \; Sid Ahmed Ould Beinane and Sid Ahmed ould  Ahmed Mahmoud   ]{ Samir Al Mohammady,\; Sid Ahmed Ould Beinane and Sid Ahmed o.Ahmed Mahmoud }

\address{ Samir Almohammady  \endgraf
  Mathematics Department, College of Science, Jouf
University,\endgraf Sakaka P.O.Box 2014. Saudi Arabia}
\email{senssar@ju.edu.sa }

\address{ Sid Ahmed Ould Beinane  \endgraf
  Mathematics Department, College of Science, Jouf
University,\endgraf Sakaka P.O.Box 2014. Saudi Arabia}
\email{Beinane06@gmail.com}

 \address{ Sid Ahmed Ould Ahmed Mahmoud  \endgraf
  Mathematics Department, College of Science, Jouf
University,\endgraf Sakaka P.O.Box 2014. Saudi Arabia}
\email{sidahmed@ju.edu.sa,\;sidahmed.sidha@gmail.com \endgraf
{\bf Lab: Mathematical Analysis and Applications}}

\maketitle
\begin{abstract}

The purpose of the paper is to  introduce  and study a new class
of   operators on semi-Hilbertian spaces i.e.; spaces generated by positive semidefinite
sesquilinear forms. Let   ${\mathcal H}$  be a Hilbert space and let $A$ be a positive bounded operator on ${\mathcal H}$. The semi-inner product   $\left\langle h\;|\;k\right\rangle_A:=\left\langle Ah\;|\;k\right\rangle$\; $h,k \in {\mathcal H}$  induces a semi-norm $\|.\|_A$.  This makes   ${\mathcal H}$  into a semi-Hilbertian space. An operator $T\in {\mathcal B}_A({\mathcal H})$ is said to be $(n,m)$-$A$-normal if $\big[T^n,  \big(T^{\sharp_A}\big)^m\big]:=T^n\big(T^{\sharp_A}\big)^m-\big(T^{\sharp_A}\big)^mT^n=0$ for some positive integers $n$ and $m$.
\end{abstract}

\

\begin{center}
{\section{{\bf Introduction and preliminaries}}}
\end{center}

\

\
Throughout this paper, let
   $\big({\mathcal H},\;\big\langle .\;|\;.\big\rangle\big)$
be a complex Hilbert space equipped with the norm $\| .\|.$
 Let ${\mathcal B}({\mathcal H})$ denote the $C^*$-algebra of all bounded linear operators on ${\mathcal H}$
and let ${\mathcal B}({\mathcal H})^+$
 be the cone of positive operators of ${\mathcal B}({\mathcal H})$ defined as
 $${\mathcal B}({\mathcal H})^+:=\{ A\in {\mathcal B}({\mathcal H})\;\;/\; \left\langle A h\;|\;h\right\rangle \geq 0\;;\forall\; h\in {\mathcal H}\;\}.$$
\noindent For every  $T \in {\mathcal B}({\mathcal H})$  its range is denoted by ${\mathcal R}(T),$ its null space by ${\mathcal N} (T)$, and its
adjoint by $T^*$.
If ${\mathcal M}$
is a linear subspace of ${\mathcal H}$, then $\overline{{\mathcal M}}$ stands for its closure in the norm topology of
${\mathcal H}$. We denote the orthogonal projection onto a closed linear subspace ${\mathcal M}$ of ${\mathcal H}$
by $P_{_{{\mathcal M}}}.$  The positive operator $A \in  {\mathcal B}({\mathcal H})$ define a positive semi-definite sesquilinear form
$\left\langle .\;|\; .\right\rangle_A : {\mathcal H} \times {\mathcal H} \longrightarrow  \mathbb{C}$  given by
 $\left\langle h\;|\; k\right\rangle_A = \left\langle Ah\;|\; k\right\rangle$. Note that $\left\langle.\;|\; .\right\rangle_A$  defines a semi-inner
product on ${\mathcal H}$, and the semi-norm induced by it is given by $\|h\|_A=\sqrt{\left\langle h\;|\;h\right\rangle_A}$
for every $h \in {\mathcal{H}}.$  Observe that  $\|h\|_A = 0 $ if and only if  $h\in  {\mathcal N} (A)$. Then  $\|.\|_A$
is
a norm if and only if $A $ is injective, and the semi-normed space $\big( {\mathcal H}, \|.\|_A\big)$
 is a
complete space if and only if ${\mathcal R}(A)  $ is closed.\par \vskip 0.2 cm \noindent
The above semi-norm induces a semi-norm on the subspace ${\mathcal{B}}^A({\mathcal{H}})$ of ${\mathcal{B}}({\mathcal{H}}) $ consisting
of all  $T \in {\mathcal{B}}({\mathcal{H}})$ so that for some $c > 0$ and for all $ h \in {\mathcal{H}}$ , $\|Th\|_A\leq  c\|h\|_A.$
Indeed, if $T \in  {\mathcal{B}}^A({\mathcal{H}})$, then
$$\|T\|_A := \sup\big\{\frac{\|Th\|_A}{\|h\|_A},\;\;h\notin {\mathcal N}(A) \;\big\}.$$
\noindent For $ T \in {\mathcal{B}}({\mathcal{H}})$, an
operator $S \in {\mathcal{B}}({\mathcal{H}})$ is called an $A$-adjoint operator of $T$ if for every $h,k \in {\mathcal  H}$, we
have  $\left\langle T h\;|\;k\right\rangle_A = \left\langle h\;|\; Sk \right\rangle_A$, that is, $AS = T^*A$. If $T$ is an $A$-adjoint of itself, then $T$ is called an $A$-selfadjoint operator.

Generally, the existence of an $A$-adjoint operator is not guaranteed. The set of all operators that admit A-adjoints
is denoted by ${\mathcal B}_A({\mathcal H})$. Note that ${\mathcal B}_A({\mathcal H})$ is a subalgebra of ${\mathcal B}({\mathcal H})$, which is neither
closed nor dense in ${\mathcal B}({\mathcal H})$. Moreover, the inclusions ${\mathcal B}_A({\mathcal H})\subseteq {\mathcal B}^A({\mathcal H})\subseteq {\mathcal B}({\mathcal H})$
hold with equality if $A$ is one-to-one and has a closed range. If  $T \in  {\mathcal B}_A({\mathcal H})$, the reduced solution of the equation
$AX = T^*A $ is a distinguished $A$-adjoint operator of $T$, which is denoted by $ T^{\sharp_A}$. Note
that $T^{\sharp_A} = A^\dag T^*A$ in which  $A^\dag $ is the Moore-Penrose inverse of $A$. It was observed that the $A$-adjoint operator $T^{\sharp_A}$
 verifies $$AT^{\sharp_A}=T^*A,\;\; {\mathcal R}(T^{\sharp_A})\subseteq \overline{{\mathcal R}(A)}\;\;\text{and}\;\;
 {\mathcal N}( T^{\sharp_A}) ={\mathcal N}(T^*A).$$   For $T, S \in {\mathcal B}_A({\mathcal H})$, it is easy to see that
 $\|T S\|_A \leq  \|T\|_A
\|S\|_A$ and $(T S)^{\sharp_A}=S^{\sharp_A}T^{\sharp_A}$.  \vskip 0.2 cm \noindent
  Notice that if $T \in  {\mathcal B}_A({\mathcal H})$, then $T^{\sharp_A} \in {\mathcal B}_A({\mathcal H})$.
$\big(T^{\sharp_A}\big)^{\sharp_A}=P_{\overline{{\mathcal R}(A)}}TP_{\overline{{\mathcal R}(A)}}$
and $\big( \big(T^{\sharp_A}\big)^{\sharp_A} \big)^{\sharp_A}=T^{\sharp_A}.$
(For more detail on the concepts cited above  see {\cite{ACG1, ACG2,ACG3}).\par \vskip 0.2 cm \noindent
\par \vskip 0.2 cm \noindent
For any arbitrary operator $T\in {\mathcal B}_A({\mathcal H})$, we can write
$$Re_A(T):=\frac{1}{2}\big(T+T^{\sharp_A}\big)\;\;\text{and}\;\;Im_A(T):=\frac{1}{2i}\big(T-T^{\sharp_A}\big).$$
The concept of $n$-normal operators as a
generalization of normal operators on Hilbert spaces has introduced and studied  by  A. S. Jibril  \cite{AJ}  and  S. A. Alzuraiqi et al.  \cite{SAP}.
 The class of $n$-power normal operators is denoted by $[n{\bf { N}}]$.  An operator $T$ is called  $n$-power normal if
 $[T^n, T^*]=0$ \big( equivalently $T^nT^*=T^*T^n$\big).
Very recently, several papers have appeared on $n$-normal operators. We refer the interested reader to \cite{MN,CLU,SMV} for complete
details.\par \vskip 0.2 cm \noindent
In \cite{EA1} and \cite{EA2}, the authors  has introduced and studied the classes of $(n,m)$-normal powers  and $(n,m)$-power quasi-normal operators as follows: An operator $T \in{\mathcal B}({\mathcal H})$ is said to be $(n,m)$-powers normal(or $(n,m)$-power normal) if $T^n\big(T^{m}\big)^*=\big(T^{m}\big)^*T^n$ and it said to be $(n,m)$-power quasi-normal if $T^n\big(T^{*}\big)^mT= \big( T^{*}\big)^mTT^n $ where $n, m$ be two nonnegative integers. We refer the interested reader to \cite{CLU} for complete details on $(n,m)$-power normal operators. \par \vskip 0.2 cm \noindent
The classes of normal, $(\alpha, \beta)$-normal, $n$-power quasi normal, isometries, partial isometries, unitary operators
etc., on Hilbert spaces have been generalized to semi-Hilbertian spaces by many authors
in many papers. (See for more detail \cite{ACG1,ACG2, ACG3, ABS,CC, JSH, AS,OBA,LS}).
An operator $T\in {\mathcal B}_A({\mathcal H})$ is said to be \par \vskip 0.2 cm \noindent
$(1)$ $A$-normal if $T^{\sharp_A}T= TT^{\sharp_A}$ \quad \cite{AS} \par \vskip 0.2 cm \noindent
$(2)$ $(\alpha, \beta)$-$A$-normal, if  $\beta^2 T^{\sharp_A}T\geq_A TT^{\sharp_A}\geq_A \alpha^2 T^{\sharp_A}T$,\;\text{for}\;$0\leq \alpha \leq 1\leq \beta$\quad \cite{ABS}\par \vskip 0.2 cm \noindent $(3)$ $(A,n)$-power-quasinormal if  $T^n(T^{\sharp_A}T)=(TT^{\sharp_A})T^n,$ \quad \cite{JSH}\par \vskip 0.2 cm \noindent $(4)$
$A$-isometry if $T^{\sharp_A}T=P_{\overline{{\mathcal R}(A)}}$ \quad \cite{ACG1}. \par \vskip 0.2 cm \noindent $(5)$  $A$-unitary if
$T^{\sharp_A}T=\big(T^{\sharp_A}\big)^{\sharp_A}T^{\sharp_A}=P_{\overline{{\mathcal R}(A)}},$ i.e.; $T$ and $T^{\sharp_A}$ are $A$-isometries \quad \cite{ACG1}.\par \vskip 0.2 cm \noindent
From now on,
$A$
denotes a positive operator on
${\mathcal H}$ i.e.;
$\big( A \in {\mathcal B}({\mathcal H})^+\big).$\par \vskip 0.2 cm \noindent
This paper is devoted to the study of some new classes of operators on Semi-Hilbert spaces
called $(n,m)$-$A$-normal operators
and $(n,m)$-$A$-quasinormal operators.
 Some properties of these classes are investigated.
\

\

\section{{\bf $(n,m)$-$A$-normal operators}}

In this section, the class of $(n,m)$-$A$-normal operators as a generalization of the
classes of $A$-normal  is introduced. In addition, we study several
properties for members from this class of operators.
\

\

\begin{definition}\label{def2.1}

Let $T\in {\mathcal B}_A({\mathcal H})$. We said that $T$ is ${(n,m)}$-$A$-normal if
\begin{equation}
\big[T^n,\big(T^{\sharp_A}\big)^m\big]=T^n \big(T^{\sharp_A}\big)^m-\big(T^{\sharp_A}\big)^mT^n=0
\end{equation}
 for some positive
integers $n$ and $m$. This class of operators will denoted by $[(n,m)
{\bf { N}}]_A
$.
\end{definition}
\begin{remark} We make the following observations:\par \vskip 0.2 cm \noindent
$(1)$ Every $A$-normal operator is an $(n,m)$-$A$-normal for all  $n ,m\in \mathbb{N}$.\par \vskip 0.2 cm \noindent $(2)$ If $n=m=1$, every $(1,1)$-$A$-normal operator is $A$-normal operator.\par \vskip 0.2 cm \noindent $(3)$
If $T\in [(1,m){\bf N}]_A,$ then $T\in [(n,m){\bf N}]_A$ and  if $T\in [(n,1){\bf N}]_A $ then $T\in [(n,m){\bf N}]_A.$\par \vskip 0.2  cm \noindent $(4)$ If $T\in [(n,m){\bf N}]_A$, then $T\in [(2n,m){\bf N}]_A\wedge [(n,2m){\bf N}]_A\wedge [(2n,2m){\bf N}]_A$.

\end{remark}
\begin{remark} In the following example we give an operator that is $(n,m)$-$A$-normal for some positive integers $n$ and $m$ but  is not a $A$-normal operator.\end{remark}
\begin{example}
Let $T=\left(
         \begin{array}{cc}
           2 & 0 \\
           1 & -2\\
         \end{array}
       \right)
$ and $A=\left(
         \begin{array}{cc}
           1 & 0 \\
           0 & 2\\
         \end{array}
       \right)$ be operators
acting on two dimensional Hilbert space $\mathbb{C}^2$.
        A simple calculation shows that $T^{\sharp_A}=\left(
         \begin{array}{cc}
           2 & 2 \\
           0 & -2\\
         \end{array}
       \right)$. Moreover  $T^{\sharp_A}T\not=TT^{\sharp_A}$ and $T^{\sharp_A}T^2=T^2T^{\sharp_A}$.
       Therefore $T$ is a $(2,1)$-$A$-normal but not a $A$-normal operator.

\end{example}
In \cite[Theorem 2.1]{AS} it was observed that if $T\in {\mathcal B}_A({\mathcal H})$ then $T$  is $A$-normal if and only if $$\|Th\|_A=\|T^{\sharp_A}h\|_A,\;\;\forall\;h\in {\mathcal H}\;\;\text{and}\;\;{\mathcal R}(TT^{\sharp_A})\subseteq \overline{{\mathcal R}(A)}.$$ In the following theorem, we generalize this characterization to $(n,m)$-$A$-normal operators.

\begin{theorem}\label{thm2.1}
Let $T\in {\mathcal B}_A({\mathcal H})$. Then
  $T$ is $(n,m)$-$A$-normal operator for some positive integers $n$ and $m$ if and only if $T$ satisfying the following conditions: \par \vskip 0.2 cm \noindent $(1)$ $\left\langle\big(T^{\sharp_A}\big)^mh\;|\;\big(T^{\sharp_A}\big)^nh\right\rangle_A=\left\langle \big(T^nh\;|\;T^mh\right\rangle_A, \;\;\;\forall\;h\in {\mathcal H}$.
  \par \vskip 0.2 cm \noindent $(2)$
 ${\mathcal R}\big(T^n\big(T^{\sharp_A}\big)^m\big)\subseteq \overline{{\mathcal R}(A)}.$
\end{theorem}
\begin{proof} Assume that $T$ is a $(n,m)$-$A$-normal operator and we need to proof that $T$ satisfying the conditions $(1)$ and $(2)$.
In fact, we have
\begin{eqnarray*}
\left\langle \big[ T^n, \big(T^{\sharp_A}\big)^m\big]h\;|\;h\right\rangle_A=0&\Rightarrow&
\left\langle T^n\big(T^{\sharp_A}\big)^mh\;|\;h\right\rangle_A-\left\langle \big(T^{\sharp_A}\big)^mT^nh\;|\;h\right\rangle_A=0\\&\Rightarrow&
\left\langle \big(T^{\sharp_A}\big)^mh\;|\;T^{*n}Ah\right\rangle-\left\langle A\big(T^{\sharp_A}\big)^mT^nh\;|\;h\right\rangle=0\\&\Rightarrow&
\left\langle \big(T^{\sharp_A}\big)^mh\;|\;\big(T^{\sharp_A}\big)^nh\right\rangle_A-\left\langle \big(T^nh\;|\;T^mh\right\rangle_A=0\\&\Rightarrow&\left\langle \big(T^{\sharp_A}\big)^mh\;|\;\big(T^{\sharp_A}\big)^nh\right\rangle_A=\left\langle \big(T^nh\;|\;T^mh\right\rangle_A.
\end{eqnarray*}
Moreover the condition $\big[ T^n, \big(T^{\sharp_A}\big)^m\big]=0$ implies  that $T^n\big(T^{\sharp_A}\big)^m=\big(T^{\sharp_A}\big)^mT^n$. \par \vskip 0.2 cm \noindent Therefore
$$ {\mathcal R}\big(T^n\big(T^{\sharp_A}\big)^m\big)={\mathcal R}\big(\big(T^{\sharp_A}\big)^mT^n\big) \subseteq {\mathcal R}\big(T^{\sharp_A}\big)\subseteq \overline{{\mathcal R}(A)}.$$
Conversely, assume that $T$ satisfying conditions $(1)$ and $(2)$ and we prove that $T$ is $(n,m)$-$A$-normal operator.
From the condition $(1)$, a simple computation shows that
\begin{eqnarray*}
\left\langle\big(T^{\sharp_A}\big)^mh\;|\;\big(T^{\sharp_A}\big)^nh\right\rangle_A-\left\langle \big(T^nh\;|\;T^mh\right\rangle_A=0&\Rightarrow&
\left\langle T^n\big(T^{\sharp_A}\big)^mh\;|\;h\right\rangle_A-\left\langle \big(T^{\sharp_A}\big)^mT^nh\;|\;h\right\rangle_A=0
\\&\Rightarrow& \left\langle \big[ T^n, \big(T^{\sharp_A}\big)^m\big]h\;|\;h\right\rangle_A=0,
\end{eqnarray*}
which implies that ${\mathcal R}\big( \big[ T^n, \big(T^{\sharp_A}\big)^m\big]\big)\subseteq {\mathcal N}(A).$
\par \vskip 0.2 cm \noindent On the other hand, if the condition $(2)$  holds, it follows that
$${\mathcal R}\big( \big[ T^n, \big(T^{\sharp_A}\big)^m\big]\big)\subseteq \overline{{\mathcal R}(A)}={\mathcal N}(A)^\bot. $$
We deduce that $\big[ T^n, \big(T^{\sharp_A}\big)^m\big]=0$  which means that the operator $T$ is $(n,m)$-$A$-normal operator.
\end{proof}

\begin{remark}
If $n=m=1$, then Theorem  \ref{thm2.1} coincides with Theorem 2.1 in \cite{AS}.
\end{remark}
The following proposition discuss the relation between $(n,m)$-$A$-normal operators and
$(m,n)$-$A$-normal operators.
\begin{proposition}\label{pro2.1}
Let $T\in {\mathcal B}_A({\mathcal H})$ such that ${\mathcal N}(A)^\bot$ is a invariant subspace of $T$. Then following statements are equivalent
\par \vskip 0.2 cm \noindent $(1)$ $T$ is $(n,m)$-$A$-normal operator.
\par \vskip 0.2 cm \noindent $(2)$ $T$ is $(m,n)$-$A$-normal operator.
\end{proposition}
\begin{proof}
$(1)\Longrightarrow (2).$ Assume that $T$ is $(n,m)$-$A$-normal operator. it follows that
$$T^n \big(T^{\sharp_A}\big)^m-\big(T^{\sharp_A}\big)^mT^n=0.$$
\begin{eqnarray*}
T^n \big(T^{\sharp_A}\big)^m-\big(T^{\sharp_A}\big)^mT^n=0&\Rightarrow&  \big(T^{\sharp_A}\big)^{\sharp_Am}\big(T^n\big)^{\sharp_A}-\big(T^n\big)^{\sharp_A}\big(T^{\sharp_A}\big)^{\sharp_Am}=0\\&\Rightarrow&
\big(P_{\overline{{\mathcal R}(A)}}TP_{\overline{{\mathcal R}(A)}}\big)^m\big(T^n\big)^{\sharp_A}-\big(T^n\big)^{\sharp_A}\big(P_{\overline{{\mathcal R}(A)}}TP_{\overline{{\mathcal R}(A)}}\big)^m=0\\&\Rightarrow&
P_{\overline{{\mathcal R}(A)}}\bigg(T^m\big(T^n\big)^{\sharp_A}-\big(T^n\big)^{\sharp_A}T^m\bigg)=0.
\end{eqnarray*}
This means that $\big(T^m\big(T^n\big)^{\sharp_A}-\big(T^n\big)^{\sharp_A}T^m\big)h \in {\mathcal N}(A)$ for all $h\in {\mathcal H}$.\par \vskip 0.2 cm \noindent
On the author hand, the fact and $ {\mathcal R}(T^{\sharp_An})\subset{\mathcal R}(T^{\sharp_A})\subset \overline{{\mathcal R}(A)}$ and the assumption that
 ${\mathcal N}(A)^\bot$ is a invariant subspace of $T$ imply that
 $\big(T^m\big(T^n\big)^{\sharp_A}-\big(T^n\big)^{\sharp_A}T^m\big)h \in \overline{{\mathcal R}(A)}$ for all $h\in {\mathcal H}$. Consequently,  $\big(T^m\big(T^n\big)^{\sharp_A}-\big(T^n\big)^{\sharp_A}T^m\big)h=0$  for all $h\in {\mathcal H}$.
 Therefore $\big[T^m,\big(T^{\sharp_A}\big)^n\big]=0$. Hence $T^{\sharp_A}$ is a $(m,n)$-$A$-normal operator.\par \vskip 0.2 cm \noindent $(2)\Rightarrow (1).$  By same way hence we omit it.

\end{proof}

It is well known that if $T\in {\mathcal B}_A({\mathcal H})$ is  $A$-normal, then $T^n$ is $A$-normal. In the following theorem, we extend this result to $(n,m)$-$A$-normal operator as follows.\par \vskip 0.2 cm \noindent
\begin{theorem}\label{th21}
Let $T\in {\mathcal B}_A({\mathcal H})$. If $T$ is a $(n,m)$-$A$-normal, then  the following statements hold: \par \vskip 0.2 cm \noindent \rm(i)  $T^j$
is $A$-normal where $j$ is the least common multiple of $n$ and $m$  i.e.; $(j=LCM(n,m))$. \par \vskip 0.2 cm \noindent \rm(ii)
$T^{nm}$ is $A$-normal operator.
\end{theorem}
\begin{proof} \rm(i)
Assume that  $T$ is a $(n,m)$-$A$-normal  that is $T^n \big(T^{\sharp_A}\big)^m=\big(T^{\sharp_A }\big)^mT^n$.
Let $j=pn$ and $j=qm$. By computation we get
\begin{eqnarray*}
T^{j}\big(T^{j}\big)^{\sharp_A}&=&T^{pn}\big(\big(T^{\sharp_A}\big)^{qm}\\&=&\big(T^n\big)^p\big(\big(T^{\sharp_A}\big)^m\big)^q
\\&=&\underbrace{T^n.\cdots .T^n}_{p-\text{times}}\underbrace{\big(T^{\sharp_A}\big)^m\cdots\big(T^{\sharp_A}\big)^m}_{q-\text{times}}\\&=&
\underbrace{\big(T^{\sharp_A}\big)^m\cdots\big(T^{\sharp_A}\big)^m}_{q-\text{times}}\underbrace{T^n.\cdots .T^n}_{p-\text{times}}\\&=&\big(T^{\sharp_A}\big)^{qm}T^{np}\\&=&\big(T^{qm}\big)^{\sharp_A}T^{np}\\&=&\big(T^j\big)^{\sharp_A}T^j,
\end{eqnarray*}
which means that $T^j$
is  $A$-normal.\par \vskip 0.2 cm \noindent \rm(ii) This statement is proved in the same way as in the statement \rm(i).
\end{proof}
\par\vskip 0.2 cm \noindent

\par\vskip 0.2 cm \noindent
\begin{proposition}\label{pro22}
Let $T\in {\mathcal B}_A({\mathcal H})$, $X=T^n+ \big(T^{\sharp_A}\big)^m$, $Y=T^n-\big(T^{\sharp_A}\big)^m$  and $Z=T^n \big(T^{\sharp_A}\big)^m$. The following statements hold:\par\vskip 0.2 cm \noindent $(1)$ $T$ is $(n,m)$-$A$-normal if and only if  $ \big[X, Y \big]=0.$\par \vskip 0.2 cm \noindent $(2)$ If $T$ is of class $[(n,m){\bf{ N}}]_A$, then  $\big[Z,X\big]=\big[Z,Y\big]=0$.
\par \vskip 0.2 cm \noindent $(3)$  $T$ is of class $[(n,m){\bf{N}}]_A$ if and only if $\big[T^n, X\big]=0$.
\par \vskip 0.2 cm \noindent $(4)$  $T$ is of class $[(n,m){\bf{ N}}]_A$ if and only if $\big[T^n, Y\big]=0$.
\end{proposition}
\begin{proof}$(1)$
\begin{eqnarray*}
&&\big[ X,Y\big]=XY-YX=0\\&\Leftrightarrow& \bigg(\big( T^n+\big(T^{\sharp_A}\big)^{m} \big)\big(T^n-\big(T^ {\sharp_A}\big)^{m}\big)\bigg)-\bigg(\big( T^n-\big(T^{\sharp_A}\big)^{m}\big)\big( T^n+\big(T^{\sharp_A}\big)^{m}
\big)\bigg)=0\\&\Leftrightarrow& T^{2n}-T^n \big(T^{\sharp_A}\big)^{m}+\big(T^{\sharp_A}\big)^{m}T^n-\big(T^{\sharp_A}\big)^{2m}\\&&-T^{2n}-T^n\big(T^{\sharp_A}\big)^{m}+\big(T^{\sharp_A}\big)^{m}T^n-\big(T^{\sharp}\big)^{2m}=0\\&\Leftrightarrow&
T^n\big(T^{\sharp_A}\big)^{m}-\big(T^{\sharp}\big)^{m}T^n=0\\&\Leftrightarrow& \big[ T^n, \big(T^{\sharp_A}\big)^m\big]=0.
\end{eqnarray*}
Hence $\big[X, Y\big]=0$ if and only if $T$ is $(n,m)$-$A$-normal.\par \vskip 0.2 cm \noindent The proof of statements $(2)$,$(3)$ and $(4)$ are straightforward.
\end{proof}

\begin{proposition}
Let $T , V \in {\mathcal B}_A({\mathcal H})$ such that ${\mathcal N}(A)$  is a reducing subspace for both $ T$ and
$V.$ If $T$ is an $(n, m)$-$A$-normal operator and $V$ is an $A$-isometry, then $V T V^{\sharp_A }$ is an
$(n, m)$-$A$-normal operator.
\end{proposition}
\begin{proof}
Since $V$ is an $A$-isometry then $V^{\sharp_A}V=P_{\overline{{\mathcal R}(A) }}$. Moreover from the fact that
${\mathcal N}(A)$  is a reducing subspace for both $ T$ and
$V$ we have $$P_{\overline{{\mathcal R}(A) }}T=TP_{\overline{{\mathcal R}(A) }},\;\;T^{\sharp_A}P_{\overline{{\mathcal R}(A) }}=P_{\overline{{\mathcal R}(A) }}T^{\sharp_A}$$ and
$$ VP_{\overline{{\mathcal R}(A)}}=P_{\overline{{\mathcal R}(A)}}V\;\;\;V^{\sharp_A}P_{\overline{{\mathcal R}(A) }}=P_{\overline{{\mathcal R}(A) }}V^{\sharp_A}.$$
It easily to check that
\begin{eqnarray*}
\big(VTV^{\sharp_A}\big)^j&=& \underbrace{\big(VTV^{\sharp_A}\big)\big(VTV^{\sharp_A}\big).\cdots \big(VTV^{\sharp_A}\big)}_{j-\text{times}}\\&=& \big(VT P_{\overline{{\mathcal R}(A)}}TV^{\sharp_A}\big)\cdots \big(VTV^{\sharp_A}\big)\\&=&P_{\overline{{\mathcal R}(A)}}VT^2V^{\sharp_A}\cdots \big(VTV^{\sharp_A}\big)\\&=&
\vdots \\&=&P_{\overline{{\mathcal R}(A)}}VT^jV^{\sharp_A}.
\end{eqnarray*}
The same arguments yield the following
\begin{eqnarray*}
\big(VTV^{\sharp_A}\big)^{\sharp_Aj}&=&\underbrace{\big(VTV^{\sharp_A}\big)^{\sharp_A}\big(VTV^{\sharp_A}\big)^{\sharp_A}\cdots \big(VTV^{\sharp_A}\big)^{\sharp_A}}_{j-\text{times}}\\&=&
\big( P_{\overline{{\mathcal R}(A)}}VP_{\overline{{\mathcal R}(A)}}T^{\sharp_A}V^{\sharp_A}\big)\cdots \big( P_{\overline{{\mathcal R}(A)}}VP_{\overline{{\mathcal R}(A)}}T^{\sharp_A}V^{\sharp_A}\big)\\&=&\vdots\\&=&
P_{\overline{{\mathcal R}(A)}} V\big(T^{\sharp_A}\big)^jV^{\sharp_A}.
\end{eqnarray*}

From the above calculation, we deduce that

\begin{eqnarray}\label{eq2.2}
\left\langle \big\{\big(VTV^{\sharp_A}\big)^{\sharp_A}\big\}^ mh\;|\;\big\{\big(VTV^{\sharp_A}\big)^{\sharp_A } \big\}^nh\right\rangle_A&=&
\left\langle P_{\overline{{\mathcal R}(A)}} V\big(T^{\sharp_A}\big)^mV^{\sharp_A}h\;|\;P_{\overline{{\mathcal R}(A)}} V\big(T^{\sharp_A}\big)^nV^{\sharp_A}h\right\rangle_A\nonumber\\&=&
\left\langle \big(T^{\sharp_A}\big)^mV^{\sharp_A}h\;|\; \big(T^{\sharp_A}\big)^nV^{\sharp_A}h\right\rangle_A
\end{eqnarray}
It is also easy
to show that
\begin{eqnarray}\label{eq2.3}
\left\langle \big(VTV^{\sharp_A}\big)^ nh\;|\;\big(VTV^{\sharp_A}\big)^mh\right\rangle_A&=&
\left\langle P_{\overline{{\mathcal R}(A)}}VT^nV^{\sharp_A}h\;|\;P_{\overline{{\mathcal R}(A)}}VT^mV^{\sharp_A}h\right\rangle_A\nonumber
\\&=&\left\langle T^nV^{\sharp_A}h\;|\;T^mV^{\sharp_A}h\right\rangle_A.
\end{eqnarray}
Since $T$ is a $(n,m)$-$A$-normal we have by combining (\ref{eq2.2})  and (\ref{eq2.3})
$$\left\langle \big\{\big(VTV^{\sharp_A}\big)^{\sharp_A}\big\}^ mh\;|\;\big\{\big(VTV^{\sharp_A}\big)^{\sharp_A } \big\}^nh\right\rangle_A=\left\langle \big(VTV^{\sharp_A}\big)^ nh\;|\;\big(VTV^{\sharp_A}\big)^mh\right\rangle_A,\;\;\;\forall\;h\in {\mathcal H}.
$$
On the other hand, we have
\begin{eqnarray*}{\mathcal R}\bigg( \big(VTV^{\sharp_A}\big)^n\big\{\big(VTV^{\sharp_A}\big)^{\sharp_A}\big\}^ m\bigg)
&=&{\mathcal R}\bigg(P_{\overline{{\mathcal R}(A)}}VT^nV^{\sharp_A} P_{\overline{{\mathcal R}(A)}} V\big(T^{\sharp_A}\big)^mV^{\sharp_A}\bigg)\\&=&
{\mathcal R}\bigg(P_{\overline{{\mathcal R}(A)}}VT^n \big(T^{\sharp_A}\big)^mV^{\sharp_A}\bigg)\\&\subseteq&\overline{{\mathcal R}(P_{\overline{{\mathcal R}(A)}})}
\\&\subseteq& \overline{{\mathcal R}(A)}.
\end{eqnarray*}
\noindent in view of Theorem \ref{thm2.1}, it follows that $VTV^{\sharp_A}$ is a $(n,m)$-$A$-normal operator.
\end{proof}

\begin{proposition}
Let  $T\in {\mathcal B}_A({\mathcal H})$  and  $S\in {\mathcal B}_A({\mathcal H})$ such that $TS=ST$ and $ST^{\sharp_A}=
T^{\sharp_A}S$. If $T$ is $(n,n)$-$A$-normal. The following statements hold:\par \vskip 0.2 cm \noindent
$(1)$ If $S$ is $A$-self adjoint, then $TS$ is $(n,n)$-$A$-normal operator.\par \vskip 0.2 cm \noindent $(2)$ If $S$ is $A$-normal operator, then $TS$ is $(n,n)$-$A$-normal operator.
\end{proposition}
\begin{proof}
$(1)$ Let $h\in {\mathcal H}$, under the assumptions that $S$  is $A$-self-adjoint $\big(AS=S^*A\big)$ and the statement $(1)$ of Theorem \ref{thm2.1} we have
\begin{eqnarray*}
\left\langle  \big(TS\big)^{\sharp_An}h\;|\;\big(TS\big)^{\sharp_An}h\right\rangle_A&=&
\left\langle  \big(S\big)^{\sharp_An} \big(T)^{\sharp_An}h\;|\; \big(S\big)^{\sharp_An} \big(T\big)^{\sharp_An}h\right\rangle_A\\&=&
\left\langle  A\big(S\big)^{\sharp_An} \big(T)^{\sharp_An}h\;|\; \big(S\big)^{\sharp_An} \big(T\big)^{\sharp_An}h\right\rangle\\&=&\left\langle  \big(S^*\big)^{n} A\big(T)^{\sharp_An}h\;|\; \big(S\big)^{\sharp_An} \big(T\big)^{\sharp_An}h\right\rangle\\&=&
\left\langle A \big(S\big)^{n} \big(T)^{\sharp_An}h\;|\; \big(S\big)^{\sharp_An} \big(T\big)^{\sharp_An}h\right\rangle\\&=&
\left\langle A  \big(T)^{\sharp_An}S^{n}h\;|\; \big(S\big)^{\sharp_An} \big(T\big)^{\sharp_An}h\right\rangle\\&=&
\left\langle   \big(T)^{\sharp_An}S^{n}h\;|\;A \big(S\big)^{\sharp_An} \big(T\big)^{\sharp_An}h\right\rangle\\&=&
\left\langle   \big(T)^{\sharp_An}S^{n}h\;|\;  \big(T\big)^{\sharp_An}S^nh\right\rangle_A\\&=&
\left\langle   T^nS^{n}h\;|\;  T^{n}S^nh\right\rangle_A\\&=&
\left\langle   \big(TS\big)^{n}h\;|\;  \big(TS\big)^nh\right\rangle_A.
\end{eqnarray*}
On the other hand, we have
$${\mathcal R}\big(  \big(TS\big)^{n} \big(TS\big)^{\sharp_An} \big)={\mathcal R}\big(T^nT^{\sharp_An}S^nS^{\sharp_An}\big)
\subseteq \overline{{\mathcal R}(A)}.$$
This means that $TS$ is a $(n,n)$-$A$-normal operator by Theorem \ref{thm2.1}.\par\vskip 0.2 cm \noindent
$(2)$ Let $S$ is a $A$-normal operator that is $SS^{\sharp_A}=S^{\sharp_A}S$ and  because $T$ is $(n,n)$-$A$-normal operator we get the following relations
\begin{eqnarray*}
\left\langle \big(ST)^{\sharp_An}h\;|\;\big(ST)^{\sharp_An}h\right\rangle_A&=&\left\langle S^{\sharp_An}T^{\sharp_An}h\;|\;
S^{\sharp_An}T^{\sharp_An}h\right\rangle_A\\&=&\left\langle AS^{\sharp_An}T^{\sharp_An}h\;|\;
S^{\sharp_An}T^{\sharp_An}h\right\rangle
\\&=&\left\langle S^{*n}AT^{\sharp_An}h\;|\;
S^{\sharp_An}T^{\sharp_An}h\right\rangle
\\&=&\left\langle T^{\sharp_An}h\;|\;
S^nS^{\sharp_An}T^{\sharp_An}h\right\rangle_A
\\&=&\left\langle T^{\sharp_An}h\;|\;
\big(S^{\sharp_A}\big)^nS^nT^{\sharp_An}h\right\rangle_A
\\&=&\left\langle S^nT^{\sharp_An}h\;|\;
S^nT^{\sharp_An}h\right\rangle_A
\\&=&\left\langle T^{\sharp_An}S^nh\;|\;
T^{\sharp_An}S^nh\right\rangle_A
\\&=&\left\langle T^{n}S^{n}h\;|\;
T^{n}S^{n}h\right\rangle_A\;\;\quad \big( \text{since}\;T \text{is}\;(n,n)-A-\text{normal}\big)
\\&=&\left\langle \big(TS\big)^nh\;|\;
\big(TS\big)^nh\right\rangle_A.
\end{eqnarray*}
On the other hand, based on  the $(n,n)$-$A$-normality of $T$ we will get the following inclusion
$${\mathcal R}\big( \big(TS\big)^n\big(TS\big)^{\sharp_An}\big)={\mathcal R}\big( T^nS^nT^{\sharp_An}S^{\sharp_An}\big)\subseteq
{\mathcal R}\big( T^nT^{\sharp_An}\big)\subseteq \overline{{\mathcal R}(A)}.$$
From the items $(1)$ and $(2)$ of Theorem \ref{thm2.1}, the operator $TS$ is $(n,n)$-$A$-normal operator.
\end{proof}
\par \vskip 0.2 cm \noindent In the following proposition, we study the relation between the two classes $[(2,m){\bf{ N}}]_A$ and $[(3,m){\bf { N}}]_A.$
\par \vskip 0.2 cm \noindent
\begin{proposition} \label{pro24} Let $T\in {\mathcal B}_A({\mathcal H})$
 such that $T\in  [(2,m){\bf {N}}]_A \wedge [(3,m){\bf {N}}]_A$  for some positive integer $m$, then $T\in [(n,m){\bf { N}}]_A$ for
all positive integer $n\geq 4$.
\end{proposition}

\begin{proof}
 It is obvious form Definition \ref{def2.1} that if  $T\in [(2,m){\bf {N}}]_A$ then $ T\in [(4,m){\bf {N}}]_A$.\par \vskip 0.2 cm \noindent
However  $T\in  [(2,m){\bf {N}}]_A \wedge [(3,m){\bf {N}}]_A$ implies that $T\in [(5,m){\bf {N}}]_A.$
\par \vskip 0.2 cm \noindent  Assume that $T\in  [(n,m){\bf {N}}]_A $  for $n\geq 5$ that is
  $$
 T^n(T^{\sharp_A})^m=(T^{\sharp_A})^mT^n.
$$
Then we have
\begin{eqnarray*}
\big[ T^{n+1},(T^{\sharp_A})^m \big]&=&T^{n+1}(T^{\sharp_A})^m-(T^{\sharp_A})^mT^{n+1}
\\&=&T(T^{\sharp_A})^mT^{n}-(T^{\sharp_A})^mT^{n+1}
\\&=&T(T^{\sharp_A})^mT^2T^{n-2}-(T^{\sharp_A})^mT^{n+1}\nonumber\\&=&T^3(T^{\sharp_A})^mT^{n-2}-(T^{\sharp_A})^mT^{n+1}\\&=&(T^{\sharp_A})^mT^{n+1}-(T^{\sharp_A})^mT^{n+1}
\\&=&0.
\end{eqnarray*}
This means that $T\in [(n+1,m){\bf{ N}}]_A.$ The proof is complete.
\end{proof}
\begin{proposition}\label{pro25}
Let $T\in {\mathcal B}_A({\mathcal H})$. If $T\in [(n,m){\bf{ N}}]_A\wedge [(n+1,m){\bf{N}}]_A$, then $T\in
 [(n+2,m){\bf{ N}}]_A$ for some positive integers $n$ and $m$. In particular $T\in [(j,m){\bf{ N}}]_A$ for all $j\geq n$.
\end{proposition}
\begin{proof}
Let $T \in [(n,m){\bf{ N}}]_A\wedge [(n+1,m){\bf{ N}}]_A$, it follows that
$$T^n\big(T^{\sharp_A}\big)^m-\big(T^{\sharp_A}\big)^mT^n=0\;\;\text{and}\;\;T^{n+1}\big(T^{\sharp_A}\big)^m-
\big(T^{\sharp_A}\big)^mT^{n+1}=0.$$
Note that
\begin{eqnarray*}
\big[T^{n+2},  \big(T^{\sharp_A}\big)^m\big]&=&T^{n+2}\big(T^{\sharp_A}\big)^m-
\big(T^{\sharp_A}\big)^mT^{n+2}\\&=& TT^{n+1} \big(T^{\sharp_A }\big)^m-\big(T^{\sharp_A}\big)^mT^{n+2}\\&=&
T\big(T^{\sharp_A}\big)^mT^{n+1}-\big(T^{\sharp_A}\big)^mT^{n+2}
\\&=&TT^n\big(T^{\sharp_A}\big)^mT-\big(T^{\sharp_A}\big)^mT^{n+2}\\&=&\big(T^{\sharp_A}\big)^mT^{n+2}-\big(T^{\sharp_A}\big)^mT^{n+2}\\&=&0.
\end{eqnarray*}
Hence $T\in [(n+2,m){\bf{ N}}]_A$. By repeating this process we can prove that $T\in[(j,m){\bf{ N}}]_A$ for all $j\geq n$.
\end{proof}

\vspace{-2mm}
\begin{proposition}\label{pro26} Let $T\in {\mathcal B}_A({\mathcal H})$.
If $T\in  [(n,m){\bf { N}}]_A\wedge (n+1,m){\bf{ N}}]_A$ such that
$T$ is on-to-one,  then $T\in [(1,m){\bf{ N}}]_A$.
\end{proposition}
\begin{proof} Let $T\in[(n,m){\bf { N}}]_A\wedge [(n+1,m){\bf { N}}]_A,$  it follows that
 $$T^n\bigg(T(T^{\sharp_A})^m-(T^{\sharp A})^mT\bigg)=0.$$ Since $T$ is
one-to-one, then so is $T^n$ and it follows that $T\big(T^{\sharp_ A}\big)^m-\big(T^{\sharp _A}\big)^mT=0$. Therefore
 $T\in [(1,m){\bf{ N}}]_A$.
\end{proof}
\begin{proposition} \label{pro27} Let $T\in {\mathcal B}_A({\mathcal H})$. The following statement are equivalent.\par \vskip 0.2 cm \noindent $(1)$
 If $T \in
[(n,2){\bf { N}}]_A\wedge [(n,3){\bf { N}}]_A$ for some positive integer $n$,  then $T\in [(n,m){\bf { N}}]_A$ for

all positive integer $m\geq 4$.\par \vskip 0.2 cm \noindent $(2)$
If $T\in [(n,m){\bf{ N}}]_A\wedge [(n,m+1){\bf{ N}}]_A$, then $T \in
 [(n,m+2){\bf{ N}}]_A$ for some

  positive integers $n,m$. In particular $T\in [(n,j){\bf{ N}}]_A$ for all $ j\geq m.$.
\end{proposition}
\begin{proof}
The proof follows by applying Proposition \ref{pro2.1} and Proposition \ref{pro24}.
\end{proof}

\begin{proposition}\label{pro29} Let $T\in {\mathcal B}_A({\mathcal H})$ .
If $T\in [(n,m){\bf { N}}]_A\wedge [(n,m+1){\bf { N}}]_A$ such that
$T^{\sharp_A}$ is one-to one,  then $T\in [(n,1){\bf{ N}}]_A=[n{\bf{ N}}]_A$.
\end{proposition}
\begin{proof} Since $T\in [(n,m){\bf { N}}]_A\wedge [(n,m+1){\bf { N}}]_A,$  it follows that
 $$(T^{\sharp_A})^m\bigg( T^nT^{\sharp_A}-T^{\sharp_A}T^n\bigg)=0.$$ If $T^{\sharp_A}$ is
one-to-one, then so is $\big(T^{\sharp_A}\big)^m$ and we obtain $T^nT^{\sharp_A}-T^{\sharp_A}T^n=0$.
Consequently $T\in [(n,1){\bf{ N}}]_A$.
\end{proof}
\par\vskip 0.2 cm \noindent
 In \cite[Theorem 2.4]{OB} it was proved that if $T$ is $(n,m)$-power normal such that $T^m$ is a partial isometry, then $T$ is a $(n+m,m)$-power normal. In the following theorem we extend this result to $(n,m)$-$A$- normal operators.\par\vskip 0.2 cm \noindent
 \begin{theorem}\label{th23}
Let $T\in {\mathcal B}_A({\mathcal H})$ be an $(n,m)$-$A$-normal for  some positive integers $n$ and $m$.The following statements hold: \par \vskip 0.2 cm \noindent $(1)$ If $n\geq m$ and $T^m\big(T^{\sharp_A }\big)^mT^m=T^m$, then $T\in \big[(n+m,m){\bf N}\big]_A$. \par \vskip 0.2 cm \noindent $(2)$ If $m\geq n$  and $\big(T^{\sharp_A }\big)^n T^n\big(T^{\sharp_A }\big)^n=\big(T^{\sharp_A }\big)^n$, then $T\in \big[(n,m+n){\bf N}\big]_A$.
 \end{theorem}
 \begin{proof} $(1)$
  Under the assumption that  $T^m\big(T^{\sharp_A }\big)^mT^m = T^m$, it follows that $$ T^m\big(T^{\sharp_A }\big)^mT^n
 =T^{n}\;\;\text{
and}\;\; T^n\big(T^{\sharp_A }\big)^mT^m=T^{n} \quad \text{for}\; n\geq m,$$  which means that   $T^n\big(T^{\sharp_A }\big)^mT^m= T^m\big(T^{\sharp_A }\big)^mT^n$. Since $T$ is a $(n,m)$-$A$ normal, we get
$$\big(T^{\sharp }\big)^mT^{n+m}=T^{n+m}\big(T^{\sharp_A}\big)^m.$$ So, $T\in \big[(m+n,m){\bf N}\big]_A$.\par \vskip 0.2 cm \noindent $(2)$ In same way,
  under the assumption that  $\big(T^{\sharp_A }\big)^nT^n\big(T^{\sharp_A }\big)^n= \big(T^{\sharp_A }\big)^n$, it follows that $$ \big(T^{\sharp_A }\big)^nT^n\big(T^{\sharp_A }\big)^m=\big(T^{\sharp_A }\big)^{m}
 \;\;\text{
and}\;\;  \big(T^{\sharp_A }\big)^mT^n\big(T^{\sharp_A }\big)^n=\big(T^{\sharp_A }\big)^{m} \quad \text{for}\; m\geq n,$$  which means that
 $$ \big(T^{\sharp_A }\big)^nT^n\big(T^{\sharp_A }\big)^m=\big(T^{\sharp_A }\big)^mT^n\big(T^{\sharp_A }\big)^n.$$
 Since $T$ is a $(n,m)$-$A$ normal, we get
$\big(T^{\sharp }\big)^{m+n}T^{n}=T^n\big(T^{\sharp_A}\big)^{n+m}.$ \par \vskip 0.2 cm \noindent So, $T\in \big[(n,m+n){\bf N}\big]_A$ and the proof is complete.
 \end{proof}
 \par\vskip 0.2 cm \noindent

\begin{proposition}

 $T\in {\mathcal B}_A({\mathcal H})$ be a $(n,m)$-$A$- normal operator for some positif integers $n$ and $m$. Then $T$ satisfying  $T^{2n}\big(T^{\sharp_A}\big)^{2m}=\bigg(T^n\big(T^{\sharp_A}\big)^m\bigg)^2.$
\end{proposition}
\begin{proof}
Since $T$ is $(n,m)$-$A$- normal operator, it follows that
\begin{eqnarray*} T^{2n}\big(T^{\sharp_A}\big)^{2m}&=&T^n T^n\big(T^{\sharp_A}\big)^m\big(T^{\sharp_A}\big)^m\\&=&
\underbrace{T^n \big(T^{\sharp_A}\big)^m}. \underbrace{T^n \big(T^{\sharp_A}\big)^m}\\&=&
\bigg(T^n \big(T^{\sharp_A}\big)^m\bigg)^2.
\end{eqnarray*}
\end{proof}

The idea of the the following proposition is inspired from \cite{BO2}.\par \vskip 0.2 cm \noindent
\begin{proposition}\label{proAA}
Let $T\in {\mathcal B}_A({\mathcal H})$ be such that $AT=TA$ and ${\mathcal N}(A)$ is a reducing subspace for $T$. If $T$ is
$n$-normal operator, then $T$ is $(n,m)$-$A$-normal operator for $m\in \mathbb{N}.$
\end{proposition}
\begin{proof}
Indeed, since $T^n$  is a normal and $T^mT^n =T^nT^m,$ it follows from Fuglede theorem([15])
that $T^{*m}T^n = T^nT^{*m}.$ Taking in consideration that under the assumptions we have
$P_{\overline{{\mathcal R}(A)}}T=TP_{\overline{{\mathcal R}(A)}}$ and    $T^{\sharp_A}=P_{\overline{{\mathcal R}(A)}}T^*.$

\begin{eqnarray*}
\big[  T^n,  \big(T^{\sharp_A}\big)^m\big]&=&T^n \big(T^{\sharp_A}\big)^m- \big(T^{\sharp_A}\big)^mT^n\\&=&
T^n\big( P_{\overline{{\mathcal R}(A)}}T^*\big)^m-\big(P_{\overline{{\mathcal R}(A)}}T^*\big)^mT^n\\&=&
P_{\overline{{\mathcal R}(A)}}\big[T^n,    T^{*m}\big]=0.
\end{eqnarray*}
Therefore $T$  is a $(n,m)$-$A$-normal.
\end{proof}
\begin{corollary}
Let $T\in {\mathcal B}_A({\mathcal H})$ be such that $AT=TA$ and ${\mathcal N}(A)$ is a reducing subspace for $T$. If $T$ is
$(n,m)$-normal operator, then $T$ is $(j,r)$-$A$-normal operator where $r \in \mathbb{N}$ and $j$ is the least common multiple of $n$ and  $m$.
\end{corollary}

\begin{proof}
Since $T$ is $(n,m)$-normal, it was observed in \cite[Lemma 4.4]{CLU} that $T^j$ is a normal operator where $j=LCM(n,m)$. By applying Proposition \ref{proAA} we get that $(j,r)$-$A$-normal operator.
\end{proof}
\section{{\bf $(n,m)$-$A$-quasinormal operators}}

\

\
In \cite{VBA}  the author has introduced the class of $(n,m)$-$A$-quasinormal operators as follows. An operator
 $T\in {\mathcal B}_A({\mathcal H})$ is said to be $(n,m)$-$A$-quasinormal if $T$ satisfying
$$ \big[  T^n, \big(T^{\sharp_A}\big)^mT\big]:=T^n\big(T^{\sharp_A}\big)^mT-\big(T^{\sharp_A}\big)^mTT^n=0,$$ for some positive integers $n$ and $m$. This class of operators will be denoted by $[(n,
m){\bf { QN}}]_A$.
\begin{remark}\label{rem32} Clearly, the class of $(n,m)$-$A$-quasinormal operators includes class of $(n,m)$- $A$-normal, i.e., the following inclusion holds $$[{(n,m)\bf { N}}]_A\subset [(n,m){\bf {  Q N}}]_A.$$
\end{remark}

We give the following example to show that there exists a $(n,m)$-$A$-quasi-normal operator which is  neither
 a $(n,m)$- $A$-normal
   for some positive integers $n$ and $m$.

\begin{example}
   Let $T$ to be the unilateral shift, that is, if ${\mathcal H} = \l^2$, the matrix
$$
T=\left( \begin{array}{cccc}
0&0&0&\ldots\\
1&0&0&\ldots\\
0&1&0&\ldots\\
\ldots&\ldots&\ldots&\ldots
\end{array}
\right)
\;\;\text{ and}\; A=I_{\l^2}\;\ (\text{the identity operator}).$$

It is easily verified that $\big[T^2, T^{\sharp_A}\big]\not=0$ and  $\big[T^2, T^{\sharp_A}T\big]=0$. So that $T$ is not a $(2,1)$-$A$ normal operator but it is a $(2,1)$-$A$ quasinormal operator.
   \end{example}
   \par \vskip 0.2 cm \noindent
The following theorem gives a characterization of $(n,m)$-$A$-quasinormal operators.
\begin{theorem}\label{thm3.1}
Let
 $T\in {\mathcal B}_A({\mathcal H})$. Then $T$  is $(n,m)$-$A$-quasinormal operator for some positive integers $n$ and $m$
 if and only if $T$ satisfying the following conditions: \par \vskip 0.2 cm \noindent $(1)$ $\left\langle\big(T^{\sharp_A}\big)^mTh\;|\;\big(T^{\sharp_A}\big)^nh\right\rangle_A=\left\langle \big(T^nTh\;|\;T^mh\right\rangle_A,\;\;\;\forall\;h\in {\mathcal H}$.
  \par \vskip 0.2 cm \noindent $(2)$
 ${\mathcal R}\big(T^n\big(T^{\sharp_A}\big)^mT\big)\subseteq \overline{{\mathcal R}(A)}.$
\end{theorem}
\begin{proof}
We omit the proof, since the techniques are similar to those
of Theorem \ref{thm2.1}.
\end{proof}

\begin{remark}
Theorem \ref{thm3.1} is an improved version of \cite[ Lemma 4.4]{VBA}.
\end{remark}

\begin{proposition}\label{pro3.1}

Let $T\in {\mathcal B}_A({\mathcal H})$ and $S\in {\mathcal B}_A({\mathcal H})$ be $(n,m)$-$A$-normal operators.
 Then their product
$ST$
is
$(n,m)$-$A$-normal operator if the following conditions are satisfied
$ST = TS$, $ST^{\sharp_A}=T^{\sharp_A} S$ and  $TS^{\sharp_A}=S^{\sharp_A} T.$
\end{proposition}

\begin{proof}
\begin{eqnarray*}
\big(TS\big)^n \big( \big(TS\big)^{\sharp_A} \big)^m\big(TS\big)&=&T^nS^n\big(T^{\sharp_A} \big)^m
\big(S^{\sharp_A}\big)^mTS\\&=&T^n\big(T^{\sharp_A} \big)^mTS^n\big(S^{\sharp_A} \big)^mS\\&=&
\big(T^{\sharp_A} \big)^mTT^n.\big(S^{\sharp_A} \big)^mSS^n\\&=& \big( \big(TS\big)^{\sharp_A} \big)^m\big(TS\big)\big(TS\big)^n.
\end{eqnarray*}
Therefore $TS$ is a $(n,m)$-$A$-quasinormal operator.
\end{proof}
\begin{remark}
Proposition \ref{pro3.1} is an improved version of \cite[ Proposition 4.5]{VBA}.
\end{remark}
\begin{proposition} \label{pro33} Let $T\in {\mathcal B}_A({\mathcal H})$ .
If $T \in [(n,m){\bf { QN}}]_A\wedge [(n+1,m){{\bf QN}}]_A$, then
 $ T\in [(n+2,m){\bf {QN}}]_A.$
\end{proposition}
\begin{proof} Assume that  $T\in [(n,m){\bf { QN}}]_A\wedge[(n+1,m){{ \bf QN}}]_A$, it follows that
\begin{equation*}
T^{n+1}\big(T^{\sharp_A}\big)^{m}T-\big(T^{\sharp_A}\big)^{m}TT^{n+1} =0 \;\;\text{and}\;\;T^{n}\big(T^{\sharp_A}\big)^{m}T-\big(T^{\sharp_A}\big)^{m}TT^{n}=0.
\end{equation*}
On the other hand, we have
\begin{eqnarray*} T^{n+2}\big(T^{\sharp_A}\big)^{m}T-\big(T^{\sharp_A}\big)^{m}TT^{n+2}&=&
T\big(T^{\sharp_A}\big)^{m}TT^{n+1}-\big(T^{\sharp_A}\big)^{m}TT^{n+2}\\&=&
 T^{n+1}\big(T^{\sharp_A}\big)^mTT-\big(T^{\sharp_A}\big)^{m}TT^{n+2}\\&=&
  \big(T^{\sharp_A}\big)^mTT^{n+2}-\big(T^{\sharp_A}\big)^{m}TT^{n+2}\\&=&0.
\end{eqnarray*}

\end{proof}
\par\vskip 0.2 cm \noindent In \cite{OB} it was proved that if $T$ is of class $[(n,m){\bf QN}]$ such that $T^m$ is a partial isometry, then $T$
is of class $[(n+m,m){\bf { QN}}]$  for $n\geq m$. We extend this result to the class of $[(n ,m){\bf { QN}}]_A$ as follows.
\par\vskip 0.1 cm
\begin{theorem}\label{th31}
Let $T\in{\mathcal B}_A({\mathcal H})$ such that $T\in [(n ,m){\bf { QN}}]$ for some positive integers $n$ and $m$. The following statements hold:\par \vskip 0.2 cm \noindent   If $T^m\big(T^{\sharp_A}\big)^mT^m=T^m$  for $n\geq m$, then $T\in [(n+m ,m){\bf { QN}}]_A. $  \par \vskip 0.2 cm \noindent
\end{theorem}
\begin{proof}
$(1)$ Assume that  $T^m$ satisfy $T^m\big(T^{\sharp_A}\big)^mT^m=T^m$ for  $m\geq n$, then we have

\begin{equation}\label{32}
T^m\big(T^{\sharp_A}\big)^{m-1}TT^m=T^m
\end{equation}
Multiplying (\ref{32}) to the left by $T^{n-m}$
and to the right by $T$
 we get
 \begin{equation}\label{33}
 T^n\big(\big(T^{\sharp_A}\big)^{m}T\big)T^m=T^{n+1}.
 \end{equation}
 Multiplying (\ref{32})  to the right by $T^{n-m+1}$ we get
  \begin{equation}\label{34}
 T^m(\big(T^{\sharp_A}\big)^mT)T^n=T^{n+1}.
 \end{equation}
 Combining (\ref{33}), (\ref{34}) and  using the fact that $T\in [(n ,m){\bf {QN}}]$ we obtain
 $$T^{n+m}(\big(T^{\sharp_A}\big)^mT)=(\big(T^{\sharp_A}\big)^mT)T^{n+m}.$$
 Therefore $T\in [(n+m,m){\bf QN}]_A$ as required.
\end{proof}
\vspace{-1mm}
\begin{proposition}\label{pro34} Let $T\in {\mathcal B}_A({\mathcal H})$, $n$ and $m$ are positive integers. The following statements hold:\par \vskip 0.2 cm \noindent $(1)$
If $T\in [(n,m){\bf{ QN}}]_A \wedge [(n+1,m){\bf{QN}}]_A$ such that
$T$ is one-to-one,  then $T\in [(1,m){\bf { QN}}]_A.$ \par \vskip 0.2 cm \noindent $2$
If
$T\in [(n,m){\bf {QN}}]_A\wedge [(n,m+1){\bf {QN}}]_A$ such that $T^*$ is one-to-one  and

 $\overline{{\mathcal R}\big(T^{\sharp_A}\big)^mT\big)}=\overline{{\mathcal R}(A)}$, then $T\in [(n,1){\bf { N}}]_A$.
\end{proposition}
\begin{proof} $(1)$ Under the assumptions that  $T\in [(n,m){\bf {QN}}]_A\wedge [(n+1,m){\bf { QN}}]_A,$  it follows that
 $$T^n\bigg(T\big(T^{\sharp_A}\big)^mT-\big(T^{\sharp_A}\big)^mTT\bigg)=0.$$  If $T$ is
injective, then so is $T^n$ and we have $T\big(T^{\sharp_A}\big)^mT-\big(T^{\sharp_A}\big)^mTT=0$.
Hence, $T\in [(1,m){\bf { QN}}]_A.$\par \vskip 0.2 cm \noindent $(2)$
Since $T \in [(n,m){\bf { QN}}]_A\wedge [(n,m+1){\bf { QN}}]_A,$ we have
\begin{eqnarray*}
&&T^n\big(T^{\sharp_A}\big)^{m+1}T-\big(T^{\sharp_A}\big)^{m+1}TT^n=0\\&\Rightarrow&
T^nT^{\sharp_A}\big(T^{\sharp_A}\big)^{m}T-T^{\sharp_A}\big(T^{\sharp_A}\big)^{m}TT^n=0\\&\Rightarrow & \bigg(T^nT^{\sharp_A}-T^{\sharp_A}T^n\bigg)\big(T^{\sharp_A}\big)^{m}T=0\\&\Rightarrow&
\bigg(T^nT^{\sharp_A}-T^{\sharp_A}T^n\bigg)\equiv 0\;\;\text{on}\;\;\overline{{\mathcal R}\big(\big(T^{\sharp_A}\big)^{m}T\big)}=\overline{{\mathcal R}(A)}.
\end{eqnarray*}
On the other hand, Since $T \in {\mathcal B}_A({\mathcal H})$, we have $T\big({\mathcal N}(A)\big)\subseteq {\mathcal N}(A)$. Moreover, by the assumption that $T^*$ is injective we obtain ${\mathcal N}(T^{\sharp_A})={\mathcal N}(A).$
If $h\in {\mathcal N}(A)$ if follows from the above observation that
$$\big(T^nT^{\sharp_A}-T^{\sharp_A}T^n\big)h=T^nT^{\sharp_A}h-T^{\sharp_A}T^nh=0.$$
Consequently, $\big(T^nT^{\sharp_A}-T^{\sharp_A}T^n\big)=0$ on ${\mathcal H}$.
Therefore $T$ is of class $[(n,1){\bf{ N}}]_A.$
\end{proof}
\vspace{-2mm}
\begin{proposition} \label{pro35} Let $T\in {\mathcal B}_A({\mathcal H})$
 such that $T \in
[(2,m){\bf{QN}}]_A\wedge [(3,m){\bf {QN}}]_A$  for some positive integer $m$, then $T\in[(n,m){\bf { QN}}]_A$ for
all positive integer $n\geq 4$.
\end{proposition}
\begin{proof} We proof the assertion by using the
mathematical induction. Since  $T \in
[(2,m){\bf{QN}}]_A\wedge [(3,m){\bf {QN}}]_A$, it follows immediately that
 \begin{equation*}\label{39}
T^4\big(T^{\sharp_A}\big)^{m}T-\big(T^{\sharp_A}\big)^{m}TT^4=0\;\;\text{and}\;T^5\big(T^{\sharp_A}\big)^{m}T-\big(T^{\sharp_A}\big)^{m}TT^5=0\;\;\text{and}\;
   \end{equation*}
Now assume that the result is true  for $n\geq 5$ that is
  \begin{equation*}
 T^n\big(T^{\sharp_A}\big)^{m}T-\big(T^{\sharp_A}\big)^{m}TT^n=0,
\end{equation*} then
\begin{eqnarray*}
 T^{n+1}\big(T^{\sharp_A}\big)^{m}T-\big(T^{\sharp_A}\big)^{m}T T^{n+1}&=&
T\big(T^{\sharp_A}\big)^{m}TT^n-\big(T^{\sharp_A}\big)^{m}TT^{n+1}\\&=&
T^3\big(T^{\sharp_A}\big)^{m}TT^{n-2}-\big(T^{\sharp_A}\big)^{m}TT^{n+1}\\&=&
\big(T^{\sharp_A}\big)^{m}TT^{n+1}-\big(T^{\sharp_A}\big)^{m}TT^{n+1}\\&=&
0.
\end{eqnarray*}
Therefore $T\in  [(n+1,m){\bf{ QN}}]_A.$  The proof is complete.
\end{proof}

\par \vskip 0.2 cm \noindent



In \cite{AS} it was observed that if $T \in {\mathcal B}_A({\mathcal H})$ is such that $TA=AT$, then $T^{\sharp_A}=PT^*.$

\par\vskip 0.2 cm \noindent Now we discuss $(n,m)$-$A$-quasinormality of an operator under some commuting conditions
on its real and imaginary part.
\begin{theorem}\label{th37}
Let $T\in {\mathcal B}_A({\mathcal H})$  such that ${\mathcal R}(T^{m-1})$ is dense. If  $TA=AT$ and ${\mathcal N}(A)$ is a reducing subspace for $T$. Then the following statements are equivalent.\par \vskip 0.2 cm \noindent  $(1)$  $T$ is of class $[(n,m){\bf { QN}}]_A$. \par \vskip 0.2 cm \noindent  $(2)$  $C_{m,\;A}$ commutes with $Re_A( T^n)$
and $Im_A(T^n)$, where $C_{m,\;A}=\sqrt{\big(T^{\sharp_A}\big)^mT^m}$.\end{theorem}
\begin{proof}
Since $T$ is a $(n,m)$-$A$-quasinormal, it follows that $$T^n\big(T^{\sharp_A}\big)^mT=\big(T^{\sharp_A}\big)^mTT^n.$$
Hence,
$$T^n\big(T^{\sharp_A}\big)^mT^m=\big(T^{\sharp_A}\big)^mT^mT^n.$$
  From the conditions that $TA=AT$ and ${\mathcal N}(A)$ is a reducing subspace for $T$, we observe that  $$TP_{\overline{\mathcal{R}(A)} }=TP_{\overline{\mathcal{R}(A)} }, T^{\sharp_A}P_{\overline{\mathcal{R}(A)} }=T^{\sharp_A}P_{\overline{\mathcal{R}(A)} } \text{ and }\; T^{\sharp_A}=P_{\overline{\mathcal{R}(A)} }T^*.$$ Therefore,
 $C_{m,\;A}$ is non-negative definite operator and by elementary calculation we get

$$C_{m,\;A}^2Re_A\big( T^n\big)=Re_A\big( T^n\big)C_{m,\;A}^2.$$
 Consequently,
 $$C_{m,\;A} Re_A \big(T^n\big) =
Re_A \big(T^n\big) C_{m,\;A}.$$  In a Similar way we can prove that $C_{m,\;A} Im_A \big(T^n\big) = Im_A \big(T^n\big)C_{m,\;A}.$ \par \vskip 0.3 cm \noindent Conversely,
assume that  $C_{m,\;A} Re_A \big(T^n \big)= Re_A \big(T^n\big)C_{m,\;A}$  and $ C_{m,\;A} Im_A \big(T^n\big)= Im_A \big(T^n\big)C_{m,\;A}.$   Then $$C_{m,\;A}^2 Re_A \big(T^n \big)= Re_A \big(T^n\big) C_{m.\;A}^2
\;\;\;\text{and }\;\;  C_{m,\;A}^2 Im_A \big(T^n\big) = Im_A \big(T^n \big)C_{m,\;A}^2.$$  Hence
$$C_{m,\;A}^2 \big(Re_A \big(T^n\big)+iIm_A \big(T^n\big)\big)=\big( Re_A \big(T^n\big)+iIm_A \big(T^n\big)\big)C_{m,\;A}^2,$$ and therefore $$C_{m.\;A}^2T^n=T^nC_{m,\;A}^2.$$\par \vskip 0.2 cm \noindent   On the other hand, we have

\begin{eqnarray*}
C_{m,\;A}^2T^n=T^nC_{m,\;A}^2&\Leftrightarrow & \big(T^{\sharp_A}\big)^mT^mT^n-T^n\big(T^{\sharp_A}\big)^mT^m=0\\&\Leftrightarrow&
\bigg( \big(T^{\sharp_A}\big)^mTT^n-T^n\big(T^{\sharp_A}\big)^mT\bigg)T^{m-1}=0\\&\Leftrightarrow& \big(T^{\sharp_A}\big)^mTT^n-T^n\big(T^{\sharp_A}\big)^mT=0\;\;\;
\bigg( \overline{{\mathcal R}(T^{m-1})} ={\mathcal H}\bigg).
\end{eqnarray*}
Therefore $T$ is of class $[(n,m){\bf{ QN}}]_A$.
\end{proof}
\vspace{-1mm}

\par\vskip 0.2 cm \noindent
\begin{theorem}\label{th2.2}

Let $T\in {\mathcal B}_A({\mathcal H})$  such that ${\mathcal R}(T^{m-1})$ is dense,  $TA=AT$ and ${\mathcal N}(A)$ is a reducing subspace for $T$. If  $T$  satisfied the following conditions\par \vskip 0.2  cm \noindent \rm(i)
$B_{m,\;A}$  commutes with $Re_A(T^m)$ and  $Im_A( T^m)$. \par \vskip 0.2  cm \noindent \rm(ii) $C_{m,\;A}^2T^n=T^nB_{m,\;A}^2,$  where $B_{m,\;A}=\sqrt{T^m\big(T^{\sharp_A}\big)^m}.$
\par \vskip 0.2  cm \noindent Then $T$ is $(m,m)$-$A$-quasinormal operator.
\end{theorem}
\begin{proof}
Since $B_{m,\;A}Re_A(T^m)= Re_A(T^m)B_{m,\;A}$ and  $B_{m,\;A} Im_A (T^m)= Im_A (T^m) B_{m,\;A}$ , it follows that
$$\left\{\begin{array}{c}
           B_{m,\;A}^2T^m+B^2\big(T^m\big)^{\sharp_A}=T^mB_{m,\;A}^2+\big(T^m\big)^{\sharp_A}B_{m,\;A}^2 \\
           \\
           B_{m,\;A}^2T^m-B_{m,\;A}^2\big(T^m\big)^{\sharp_A}=T^mB_{m,\;A}^2-\big(T^m\big)^{\sharp_A}B_{m,\;A}^2
         \end{array}\right.
.$$
This gives $$B_{m,\;A}^2T^m = T^mB_{m,\;A}^2 = C_{m,\;A}^2T^m.$$
On the other hand, we have
\begin{eqnarray*}
B_{m,\;A}^2T^m = C_{m,\;A}^2T^m &\Rightarrow& T^m\big(T^{\sharp_A}\big)^mT^m-\big(T^{\sharp_A}\big)^mT^mT^m=0\\
&\Rightarrow& \bigg( T^m\big(T^{\sharp_A}\big)^mT-\big(T^{\sharp_A}\big)^mTT^m  \bigg)T^{m-1}=0\\&\Rightarrow&
 T^m\big(T^{\sharp_A}\big)^mT-\big(T^{\sharp_A}\big)^mTT^m =0\;\;\text{on}\;\;\overline{{\mathcal{R}(T^{m-1})}} ={\mathcal H}.
\end{eqnarray*}
Therefore $ T^m\big(T^{\sharp_A}\big)^mT-\big(T^{\sharp_A}\big)^mTT^m =0$ and so that $T$ is  a$(m,m)$-$A$-quasinormal operator.
\end{proof}

\begin{proposition}
Let $T\in {\mathcal B}_A({\mathcal H})$ be a $(n,m)$-$A$-quasinormal, then
$$\big(T^{\sharp_A}\big)^{2m}T^{2n}=\bigg( \big(T^{\sharp_A}\big)^{m}T^{n}  \bigg)^2.$$
\end{proposition}
\begin{proof}
Since $T$ is a $(n,m)$-$A$-quasinormal, it follows that $T^n\big(T^{\sharp_A}\big)^mT=\big(T^{\sharp_A}\big)^mTT^n$.
On the other hand, we have
\begin{eqnarray*}\big(T^{\sharp_A}\big)^{2m}T^{2n}&=&\big(T^{\sharp_A}\big)^{m} \big(T^{\sharp_A}\big)^{m}T^{n}T^{n}
\\&=&\big(T^{\sharp_A}\big)^{m} \big(T^{\sharp_A}\big)^{m}TT^{n}T^{n-1}\\&=&
\big(T^{\sharp_A}\big)^{m} T^n\big(T^{\sharp_A}\big)^{m}T^{n}=\bigg( \big(T^{\sharp_A}\big)^{m}T^{n}  \bigg)^2.
\end{eqnarray*}
\end{proof}

\end{document}